\documentclass[10pt]{article}

\usepackage{graphics}
\usepackage{graphicx}

\usepackage[a4paper, left=35mm,right=35mm,top=34mm,bottom=34mm]{geometry}
\usepackage[utf8]{inputenc}
\usepackage[T1]{fontenc}
\usepackage[english]{babel}

\usepackage{enumerate}
\usepackage{graphicx}
\usepackage{hyperref}
\hypersetup{
    colorlinks=true,
    linkcolor=blue,
    filecolor=magenta,      
    urlcolor=cyan,
}
\usepackage{lipsum}
\usepackage{amsfonts}
\usepackage{graphicx}
\usepackage{epstopdf}
\usepackage{algorithmic}

 \usepackage{dsfont}
 \usepackage{psfrag}
 \usepackage{color}
 \usepackage{tikz}
 \usepackage{subfigure}
\usepackage{mathtools,amsthm,amssymb,amsfonts}
\usepackage{algorithm}
\usepackage{algorithmic}
\makeatother
\theoremstyle{plain}
\def\del  {\partial}
\def\eps{\varepsilon}

\def\R{\mathbb{R}}

\def\N{\mathbb{N}}
\def\C{\mathbb{C}}

 \def\dx{{\rm d}x}
 \def\dy{{\rm d}y}

\newtheorem{proposition}{\textbf{Proposition}}
\newtheorem{corollary}{\textbf{Corollary}}
\newtheorem{remark}{\textbf{Remark}}
\newtheorem{theorem}{\textbf{Theorem}}

\newtheorem{definition}{\textbf{Definition}}

\usepackage{caption} 
\usepackage{fancyhdr}

\author{
  {\normalsize Anna Rozanova-Pierrat}\thanks{CentraleSup\'elec, Universit\'e Paris-Saclay, France
    (correspondence, anna.rozanova-pierrat@centralesupelec.fr).}
		}
\title{Generalization of Rellich-Kondrachov theorem and trace compacteness in the framework of irregular and fractal boundaries}
\date{}

\begin{document}
\maketitle
\thispagestyle{fancy}

\begin{abstract}
\noindent We present a survey of recent results of the functional analysis allowing to solve PDEs in a large class  of domains with irregular boundaries. We extend the previously introduced concept of admissible domains with a $d$-set boundary on the domains with the boundaries on which the measure is not necessarily Ahlfors regular $d$-measure. This gives a generalization of Rellich-Kondrachov theorem and  the compactness of the trace operator, allowing to obtain,
as for a regular classical case
the unicity/existence of weak solutions of Poisson boundary valued problem with the Robin boundary condition and to obtain the usual properties of the associated spectral problem. 

%
%
\end{abstract}


\section{Introduction}
From the theory of the partial differential equations it is known that the irregularity  of the boundary of the considered domain can be a serious obstacle even for the proof of the existence of a weak solution. In this paper we are interesting in the question which is the worst boundary or a class of boundaries for which we still have the weak well posedness of the elliptic problems. 
For fixing a typical example we chose to work with the Poisson equation with the homogeneous Robin boundary condition
\begin{equation}\label{LaplaceeqRobin1}
\left\lbrace
\begin{array}{l}
-\Delta u=f \hbox{ in }\Omega,\\
\frac{\partial u}{\partial\nu}+\alpha u=0 \hbox{ with }\alpha>0 \hbox{ on }\partial\Omega.
\end{array}
\right.
\end{equation}
Thus the general approach is to start to find the weak formulation of this  problem.
Hence, it is important to be able integrate by parts and to work with the trace operator on $\del \Omega$.
For at least Lipschitz $\del \Omega$ it is  classical and well-known (for sufficiently smooth boundary see Raviart-Thomas~\cite{RAVIART-1983}, for the Lipschitz case see Marschall~\cite{MARSCHALL-1987} and~\cite{GRISVARD-1974, NECAS-1967}).
 
 If $\del \Omega$ is Lipschitz, then the normal unit vector $\nu$ to the boundary $\del \Omega$ exists almost everywhere, the trace operator $\mathrm{Tr}: H^1(\Omega)\to H^{\frac{1}{2}}(\del \Omega)$ is linear continuous and surjective~\cite{LIONS-1972,MARSCHALL-1987,GRISVARD-1974,NECAS-1967} with a linear continuous right inverse, $i.e.$ the extension operator $$E: H^{\frac{1}{2}}(\del \Omega) \to H^1(\Omega) \quad \hbox{is such that }\mathrm{Tr}(E(u))=u.$$
 Moreover, for $u$, $v\in H^1(\Omega)$ with $\Delta u \in L^2(\Omega)$ it holds the usual Green formula in the following sense 
 \begin{equation}\label{EQintegByPartsH05}
 \int_\Omega \nabla u v\dx=\langle \frac{\del u}{\del \nu}, \mathrm{Tr}v\rangle _{((H^\frac{1}{2}(\del \Omega))', H^\frac{1}{2}(\del \Omega))}
-\int_\Omega \nabla v \nabla u \dx.	
 \end{equation}
  This formula understands the existence of the normal derivative of $u$ on $\del \Omega$ as the existence of a linear continuous form on $H^\frac{1}{2}(\del \Omega)$, where $H^\frac{1}{2}(\del \Omega)$ is the image of $H^1(\Omega)$ for a Lipschitz domain $\Omega$ by the trace operator. The dual space $(H^\frac{1}{2}(\del \Omega))'$ is usually denoted by $H^{-\frac{1}{2}}(\del \Omega)$.
  
  In this weak way for Lipschitz domains it is also possible to define the operator of divergence for vector valued functions (see for instance Theorem~2.5 $\S~2$~\cite{GIRAULT-1986}) or simply the usual integration by parts  for all $u$ and $v$ from $H^1(\Omega)$ in the following weak sense
      \begin{equation}\label{IPP}
 \langle u \nu_i,v\rangle_{(H^{-\frac{1}{2}}(\del \Omega), H^\frac{1}{2}(\del \Omega))}:= \int_\Omega \frac{\del u}{\del x_i} v\dx+\int_{\Omega} u\frac{\del v}{\del x_i}\dx \quad i=1,\ldots,n,
\end{equation}
where by $u \nu_i$ is denoted the linear continuous functional on $H^\frac{1}{2}(\del \Omega)$.
 
 In the same time thanks to the classical results of Calderon-Stein~\cite{CALDERON-1961,STEIN-1970} it is known that every Lipschitz domain $\Omega$ is an extension domain for the Sobolev space $W_p^k(\Omega)$ with $1\le p\le \infty$, $k\in \N^*$, which means
 %
\begin{definition}\textbf{($W_p^k$-extension domains)}
 A domain $\Omega\subset \R^n$ is called a $W_p^k$-extension domain ($k\in \N^*$) if there exists a bounded linear extension operator $E: W_p^k(\Omega) \to W^k_p(\R^n)$. This means that for all $u\in W_p^k(\Omega)$ there exists a $v=Eu\in  W^k_p(\R^n)$ with $v|_\Omega=u$ and it holds
 $$\|v\|_{W^k_p(\R^n)}\le C\|u\|_{W_p^k(\Omega)}\quad \hbox{with a constant } C>0.$$
\end{definition}

This result was generalized by Jones~\cite{JONES-1981} in the framework of $(\eps,\infty)$-domains which give an optimal class of extension domains in $\R^2$, but not in $\R^3$. More recently the optimal class of extension domains for $p>1$ in $\R^n$ is found by~Haj\l{}as, Koskela and Tuominen~\cite{HAJLASZ-2008}. These results are discussed in Section~\ref{SecExtDom}, where we give all definitions.

  Thanks to works~\cite{WALLIN-1991,JONSSON-1997,LANCIA-2002,BARDOS-2016,ARFI-2017} it is possible to generalize the trace operator for more irregular boundaries, as for instance the $d$-sets or even on sets without a fixed dimension~\cite{JONSSON-2009,HINZ-2020}. The definition of the trace for a regular distribution and different image spaces  giving different Green formulas are presented in Section~\ref{secTraceGreen}. 
  
  But to be able to ensure the weak well-posedness of problem~(\ref{LaplaceeqRobin1}) and also for the associated spectral problem of $-\Delta$, we  need to have in addition the compactness of the inclusion $H^1(\Omega)$ in $L^2(\Omega)$ and the compactness of the trace operator this time considered as an operator from $H^1(\Omega)$ to $L^2(\del \Omega)$.
 Thanks to~\cite{EDMUNDS-1987} Theorem V.4.17 it is known that if a domain $\Omega$ has a continuous boundary (in the sense of graphs, see~\cite{EDMUNDS-1987} Definition V.4.1) then $H^1(\Omega)$ is compactly embedded in $L^2(\Omega)$. The general $d$-set boundaries with $d>n-1$, as for instance a von Koch curve, does not satisfy the assumption to have a continuous boundary. In our article~\cite{ARFI-2017} this fact was proven in the framework of admissible domains with a $d$-set boundary. Here we prove it also for more general boundaries as in~\cite{JONSSON-1994,JONSSON-2009} (see Section~\ref{SecTh-RK}). We hence update the concept of admissible domains firstly introduced in~\cite{ARFI-2017} following the same idea to introduce the class of all Sobolev extension domains with boundaries on which it is possible to define a surjective linear continuous trace operator with linear continuous right inverse. To insist on their extension nature, we thus call these domains Sobolev admissible domains (see Definition~\ref{DefAdmis}).
 
 The most common examples of Sobolev admissible domains are domains with regular or Lipschitz boundaries, with a $d$-set boundaries as Von Koch fractals or with a ``mixed'' boundary presented for instance by a three-dimensional cylindrical domain constructed on a base of a two-dimensional domain with a $d$-set boundary  as considered for the Koch snowflake base in~\cite{LANCIA-2010,ARXIV-CREO-2018}. 
 
 The generalization of the Kondrachov-Rellich theorem in the framework of Sobolev admissible domains allows to extend the compactness studies of the trace from~\cite{ARENDT-2011} and to update the results of~\cite{ARFI-2017} (see Section~\ref{secCompTr}): for a Sobolev admissible domain with a compact boundary the trace operator considered from $H^1(\Omega)$ to $L^2(\del \Omega)$ is compact.
 
 Thus, as for the usual Lipschitz bounded case, the problem~(\ref{LaplaceeqRobin1}) is weakly well-posed and the corresponding spectral problem have a countable number of eigenvalues going to $+\infty$ with the eigenfunctions forming an orthogonal basis in $H^1(\Omega)$ which becomes an orthonormal basis in $L^2(\Omega)$ by the classical Hilbert-Schmidt theorem for compact auto-adjoint operators on a Hilbert space (see Section~\ref{secApplic}).
 
 To summarize, the rest of the paper is organized as follows. In Section~\ref{SecExtDom}
we present recent results on Sobolev extension domains. In Section~\ref{secTraceGreen} we firstly define the trace operator on a $d$-set in sub-Section~\ref{ss-Dset}  and secondly in sub-Section~\ref{ss-Gen} we give the analogous results in a more abstract measure framework which are not necessarily $d$-dimensional. We finish the section by a generalization of the Green formula and of the formula of the integration by parts for the abstract measure framework in sub-Section~\ref{ss-Green}. Using the results on the trace and on the extension operators, we introduce the concept of Sobolev admissible domains in Section~\ref{SecTh-RK} and generalize the Rellich-Kondrachov theorem. In Section~\ref{secCompTr}  we continue the generalization and show  the compactness of the trace operator considered this time as an operator mapping not on its image, but in $L_p(\del \Omega)$. Section~\ref{secApplic} gives an example of the application of obtained theorems by showing the well-posedness of the Poisson problem~(\ref{LaplaceeqRobin1}) on the $H^1$-Sobolev admissible domains with a standard notation $W^1_2(\Omega)=H^1(\Omega)$.

\section{Sobolev extension domains}\label{SecExtDom}
Following~\cite{ARFI-2017}, let us start by recalling the classical results of Calderon-Stein~\cite{CALDERON-1961,STEIN-1970}:
every Lipschitz domain $\Omega$ is an extension domain for $W_p^k(\Omega)$ with $1\le p\le \infty$, $k\in \N^*$.
This result was generalized by Jones~\cite{JONES-1981} in the framework of $(\eps,\delta)$-domains:
\begin{definition}\label{DefEDD}\textbf{($(\eps,\delta)$-domain~\cite{JONES-1981,JONSSON-1984,WALLIN-1991})}
An open connected subset $\Omega$ of $\R^n$ is an $(\eps,\delta)$-domain, $\eps > 0$, $0 < \delta \leq \infty$, if whenever $x, y \in \Omega$ and $|x - y| < \delta$, there is a rectifiable arc $\gamma\subset \Omega$ with length $\ell(\gamma)$ joining $x$ to $y$ and satisfying
\begin{enumerate}
 \item $\ell(\gamma)\le \frac{|x-y|}{\eps}$  (thus locally quasiconvex) and
 \item $d(z,\del \Omega)\ge \eps |x-z|\frac{|y-z|}{|x-y|}$ for $z\in \gamma$.
\end{enumerate}
\end{definition}
As the constant $\delta$ is allowed to be equal to $+\infty$ it is possible to avoid the local character of this definition and in this case just to say $(\eps,\infty)$-domain. Definition~\ref{DefEDD} without assumption~2 gives the definition of a locally quasiconvex domain. Assumption~2 does not allow the boundary to collapse into an infinitely thing structures, as for instance happens in the fractal threes. Actually it is the reason why the fractal threes~\cite{ACHDOU-2013} are not $(\eps,\infty)$-domains.

The $(\eps,\delta)$-domains  are also called locally uniform domains~\cite{HERRON-1991}.
Actually, bounded locally uniform domains, or  bounded $(\eps,\delta)$-domains, are equivalent (see~\cite{HERRON-1991} point 3.4) to the uniform domains, firstly defined by Martio and Sarvas in~\cite{MARTIO-1979}, for which there are no more restriction $|x-y|<\delta$  (see Definition~\ref{DefEDD}).

Thanks to Jones~\cite{JONES-1981}, it is known that any $(\eps,\delta)$-domain in $\R^n$ is a $W_p^k$-extension domain for all $1\le p\le\infty$ and $k\in \N^*$. Moreover, for a bounded finitely connected domain $\Omega\subset \R^2$, Jones~\cite{JONES-1981} proved that
\begin{multline*}
	\Omega \hbox{ is a } W_p^k\hbox{-extension domain } (1\le p\le\infty \hbox{ and }k\in \N^*) \Longleftrightarrow \\
	\Omega \hbox{ is an } (\eps,\infty)\hbox{-domain for some } \eps>0 \Longleftrightarrow\\
	\hbox{ the boundary }\del \Omega \hbox{ consists of finite number of points and quasi-circles.}
\end{multline*}

However, it is no more true for $n\ge3$, $i.e.$ there are $W_p^1$-extension domains which are not locally uniform~\cite{JONES-1981} (in addition, an $(\eps,\delta)$-domain in $\R^n$ with $n\ge 3$ is not necessary a quasi-sphere).

To discuss general properties of  locally uniform domains, let us introduce Ahlfors $d$-regular sets or $d$-sets:
\begin{definition}\label{Defdset}\textbf{(Ahlfors $d$-regular set or $d$-set~\cite{JONSSON-1984,JONSSON-1995,WALLIN-1991,TRIEBEL-1997})}
Let $F$ be a closed Borel non-empty subset of $\R^n$. The set $F$ is is called a $d$-set ($0<d\le n$) if there exists a $d$-measure  $\mu$ on $F$, $i.e.$ a positive Borel measure with support $F$ ($\operatorname{supp} \mu=F$) such that there exist constants 
$c_1$, $c_2>0$,
\begin{equation*}
 c_1r^d\le \mu(\overline{B_r(x)})\le c_2 r^d, \quad \hbox{ for  } ~ \forall~x\in F,\; 0<r\le 1,
 \end{equation*}
where $B_r(x)\subset \R^n$ denotes the Euclidean ball centered at $x$ and of radius~$r$.
\end{definition}
As~\cite[Prop.~1, p~30]{JONSSON-1984} all $d$-measures on a fixed $d$-set $F$ are equivalent, it is also possible to define a $d$-set by the $d$-dimensional Hausdorff measure $m_d$:
 \begin{equation*}
 c_1r^d\le m_d(F\cap \overline{B_r(x)})\le c_2 r^d, \quad \hbox{ for  } ~ \forall~x\in F,\; 0<r\le 1
 \end{equation*}
 which in particular implies that $F$ has Hausdorff dimension $d$ in the neighborhood of each point of $F$~\cite[p.33]{JONSSON-1984}.

If the boundary $\del \Omega$ is a $d$-set endowed with the $d$-dimensional Hausdorff measure restricted to $\del \Omega$, then we denote by $L_p(\del \Omega, m_d)$ the Lebesgue space defined with respect to this measure with the norm
$$\|u\|_{L_p(\del \Omega,m_d)}=\left(\int_{\del \Omega} |u|^p d m_d \right)^\frac{1}{p} .$$

From~\cite{WALLIN-1991}, it is known that
\begin{itemize}
 \item All $(\eps,\delta)$-domains in $\R^n$ are $n$-sets ($d$-set with $d=n$):
 $$\exists c>0\quad \forall x\in \overline{\Omega}, \; \forall r\in]0,\delta[\cap]0,1] \quad \lambda(B_r(x)\cap \Omega)\ge C\lambda(B_r(x))=cr^n,$$
 where $\lambda(A)$ denotes the Lebesgue measure of a set $A$ in $\R^n$. This property is also called the measure density condition~\cite{HAJLASZ-2008}. Let us notice that an $n$-set
$\Omega$ cannot be ``thin'' close to its boundary $\del \Omega$, since it must all times contain a non trivial ball in its neighborhood.
 \item If $\Omega$ is  an $(\eps,\delta)$-domain and $\del \Omega$ is a $d$-set ($d<n$) then $\overline{\Omega}=\Omega\cup \del \Omega$ is an $n$-set.
\end{itemize}
In particular, a Lipschitz domain $\Omega$ of $\R^n$ is an $(\eps,\delta)$-domain and also an $n$-set~\cite{WALLIN-1991}. But not every $n$-set is an $(\eps,\delta)$-domain: adding an in-going cusp to an $(\eps,\delta)$-domain we obtain an $n$-set which is not an $(\eps,\delta)$-domain anymore.
Self-similar fractals (e.g., von Koch's snowflake domain) are examples of $(\eps,\infty)$-domains with the $d$-set boundary~\cite{CAPITANELLI-2010,WALLIN-1991}, $d>n-1$.


Recently, Haj\l{}asz, Koskela and Tuominen~\cite{HAJLASZ-2008} have proved that every $W_p^k$-extension domain in $\R^n$ for $1\le p <\infty$ and $k\ge 1$, $k\in \N$ is an $n$-set.
In addition they proved the following statements:
\begin{theorem}\label{ThHajlasz}
	\begin{enumerate}
  \item[(i)] A domain $\Omega\subset \R^n$ is a $W^1_\infty$-extension domain if and only if $\Omega$ is uniformly locally quasiconvex.
 \item[(ii)] For $1<p <\infty$, $k=1,2,...$ a domain $\Omega\subset \R^n$ is a $W^k_p$-extension domain if and only if  $\Omega$ is an $n$-set and $W_p^k(\Omega)=C_p^k(\Omega)$ (in the sense of equivalent norms).
\end{enumerate}
\end{theorem}

By $C_p^k(\Omega)$ is denoted the space of the fractional sharp maximal functions:
\begin{definition}
For a set $\Omega\subset \R^n$ of positive Lebesgue measure, \begin{multline*}
                                                              C_p^k(\Omega)=\{f\in L_p(\Omega)|\\
                                                              f_{k,\Omega}^\sharp(x)=\sup_{r>0} r^{-k}\inf_{P\in \mathcal{P}^{k-1}}\frac{1}{\lambda(B_r(x))}\int_{B_r(x)\cap \Omega}|f-P|\dy\in L^p(\Omega)\}
                                                             \end{multline*}
with the norm $\|f\|_{C_p^k(\Omega)}=\|f\|_{L_p(\Omega)}+\|f_{k,\Omega}^\sharp\|_{L_p(\Omega)}.$ By $\mathcal{P}^{k-1}$ is denoted the space of polynomials of the order $k-1$.
\end{definition}

From~\cite{JONES-1981} and~\cite{HAJLASZ-2008} we directly have~\cite{ARFI-2017}
\begin{corollary}
 Let $\Omega$ be a bounded finitely connected domain in $\R^2$ and $1<p<\infty$, $k\in \N^*$. The domain $\Omega$ is a $2$-set with $W_p^k(\Omega)=C_p^k(\Omega)$ (with norms' equivalence) if and only if $\Omega$ is an $(\eps,\delta)$-domain and its boundary $\del \Omega$ consists of a finite number of points and quasi-circles.
\end{corollary}

The question about $W^k_p$-extension domains is equivalent to the question of the continuity of the trace operator  $\mathrm{Tr}: W^k_p(\R^n) \to W^k_p(\Omega)$, the trace operator on the domain $\Omega$.
In the next section we introduce the notion of the trace on any Borel set which we use taking the trace to the boundary.
\section{Trace on the boundary and Green formulas}\label{secTraceGreen}

\subsection{Framework of $d$-sets and Markov's local inequality}\label{ss-Dset}

From~\cite{JONSSON-1984} p.39, it is also known that all closed $d$-sets with $d>n-1$ preserve Markov's local inequality:
\begin{definition}\textbf{(Markov's local inequality)}
A closed subset $V$ in $\R^n$ preserves Markov's local inequality if for every fixed  $k\in \N^*$, there exists a constant $c=c(V,n,k) > 0$, such that
$$
\max_{V\cap \overline{B_r(x)}} |\nabla P | \le \frac{c}{r}\max_{V\cap \overline{B_r(x)}}|P|
$$
for all polynomials $P \in \mathcal{P}_k$ and all closed balls $\overline{B_r(x)}$, $x \in V$ and $0 < r \le 1$.
\end{definition}

For instance, self-similar sets that are not  subsets of any $(n-1)$-dimensional subspace of $\R^n$, the closure of a domain $\Omega$ with Lipschitz boundary and also $\R^n$ itself preserve Markov's local inequality (see Refs.~\cite{JONSSON-1997,WALLIN-1991}).
The geometrical characterization of sets preserving Markov's local inequality was initially given in~\cite{JONSSON-1984-1} (see Theorem 1.3) and can be simply interpreted as sets which are not too flat anywhere. It can be  illustrated by the following theorem of Wingren~\cite{WINGREN-1988}:
\begin{theorem}
A closed subset $V$ in $\R^n$ preserves Markov's local inequality if and only if there exists a constant $c>0$ such that for every ball
$B_r(x)$ centered in $x\in V$ and with the radius $0 < r \le 1$, there are $n + 1$
affinely independent points $y_i \in V\cap B_r(x)$, $i=1,\ldots,n+1$, such that the $n$-dimensional ball
inscribed in the convex hull of $y_1, y_2, \ldots, y_{n+1}$, has radius not less than $c r$.
\end{theorem}
Smooth manifolds in $\R^n$ of dimension less than $n$, as for instance a sphere, are examples of ``flat'' sets not preserving Markov's local inequality. More precisely, the sets $F$ which do not preserve Markov's inequality~\cite[Thm.~2, p.38]{JONSSON-1984} are
exactly the sets satisfying the geometric condition in the
following theorem. 
\begin{theorem}
	 A closed, non-empty subset $F$ of $\R^n$ preserves
Markov's inequality if and only if the following geometric
condition does not hold: 
for every $\eps > 0$ there exists a ball
$B_r(x_0)$, $x_0\in F$, $0< r\le 1$, so that $B_r(x_0)\cap F$  is contained in
some band of type $\{ x\in \R^n|\; (b,x-x_0)_{\R^n}< \eps r \}$, where $b\in  \R^n$,
$|b| = 1$, and $(b,x-x_0)_{\R^n}$  is the scalar product of $b$ and $x - x_0$.
\end{theorem}

The interest to work with $d$-sets boundaries preserving  Markov's inequality (thus $0<d<n$), related in~\cite{BOS-1995} with the Sobolev-Gagliardo-Nirenberg inequality, is to ensure~\cite[2.1]{WALLIN-1982} that there exists a bounded linear extension operator $\hat{E}$ of the Hölder space $C^{k-1,\alpha-k+1}(\del \Omega)$ to the Hölder space $C^{k-1,\alpha-k+1} (\R^n)$, where for $k\in \N^*$ $k-1<\alpha\le k$ (see also~\cite[p.~2]{JONSSON-1984}). This  allows  to show the existence of a linear continuous extension  of the Besov space
$B^{p,p}_{\alpha}(\del \Omega)$ on $ \del \Omega$ to the Sobolev space $W^k_p(\R^n)$ with $\alpha=k - \frac{(n - d)}{p}\ge 1$ and $k\ge 2$~\cite{JONSSON-1997}. For the extensions of minimal regularity with $k=1$, and thus with $\alpha<1$
(see in addition the definition of the Besov space Def.~3.2 in~\cite{ARXIV-IHNATSYEVA-2011}  with the help of the normalized local best approximation in the class of polynomials $P_{k-1}$ of the degree equal to $k-1$) Markov's inequality is trivially satisfied for $k=1$ on all closed sets of $\R^n$, and hence we do not need to impose it~\cite[p.~198]{JONSSON-1997}. %

But before to explain the mentioned results, let us  generalize the notion of the trace:
\begin{definition}\label{DefGTrace}
 For an arbitrary open set $\Omega$ of $\R^n$, the trace operator $\mathrm{Tr}$ is defined~\cite{JONSSON-1984} for $u\in L_1^{loc}(\Omega)$ by
$$
 \mathrm{Tr} u(x)=\lim_{r\to 0} \frac{1}{\lambda(\Omega\cap B_r(x))}\int_{\Omega\cap B_r(x)}u(y)dy,
$$
where $\lambda$ denotes the Lebesgue measure on $\R^n$.
The trace operator $\mathrm{Tr}$ is considered for all $x\in\overline{\Omega}$ for which the limit exists.
\end{definition}

Using this trace definition it holds the trace theorem on closed $d$-sets~\cite{JONSSON-1984} Ch.VII and~\cite{WALLIN-1991} Proposition~4, in which we think it is important to precise that the closed set $F$ should preserve Markov's local inequality not necessarily for all $k\in \N^*$, but at least up to $k-1$ with $k\in \N^*$, the fixed regularity of the Sobolev space of which we take the trace on $F$:
\begin{theorem}~\label{ThJWF}
   Let $F$ be a closed $d$-set preserving Markov's
 local inequality at least up to $k-1$ for a fixed $k\in \N^*$.
 If $$0 < d < n, \quad 1 < p <\infty, \quad  \hbox{and} \quad \alpha = k - \frac{(n - d)}{p} > 0,$$ then the trace operator $\mathrm{Tr}: W^k_p(\R^n)\to B^{p,p}_\alpha(F)$ is
bounded linear surjection
with a bounded right inverse $E: B^{p,p}_\alpha(F)\to W^k_p(\R^n)$, $i.e.$ $\operatorname{Tr}\circ E=\operatorname{Id}$ on $B^{p,p}_\alpha(F)$.
 \end{theorem}
 The definition of the Besov space  $B^{p,p}_{\alpha}(F)$ on a closed
$d$-set $F$ can be found, for instance, in Ref.~\cite{JONSSON-1984} p.~135 and
Ref.~\cite{WALLIN-1991}. See also Triebel for equivalent definitions~\cite{TRIEBEL-1997}.

Note that for $d=n-1$,  as it also mentioned in~\cite{BARDOS-2016}, one has $\alpha=\frac{1}{2}$ and $
B_\frac{1}{2}^{2,2}(F)=H^\frac{1}{2}(F)$ as usual in the case of the classical results~\cite{LIONS-1972,MARSCHALL-1987} for Lipschitz boundaries $\del \Omega=F$.
Since $\alpha=\frac{1}{2}<1$, as noticed previously the geometrical condition for the boundary to preserve Markov's inequality  does not occur.

Moreover, considering only $H^1(\R^n)=\{u\in L_2(\R^n)|\; \nabla u \in L_2(\R^n)\}$ we deduce from Theorem~\ref{ThJWF} that
\begin{theorem}~\label{ThH1}
   Let $F$ be a closed $d$-set,
  $$0\le n-2 < d < n, \quad \hbox{and} \quad \alpha = 1 - \frac{(n - d)}{2} > 0,$$
  then the trace operator $\mathrm{Tr}: H^1(\R^n)\to B^{2,2}_\alpha(F)$ is
bounded linear surjection
with a bounded right inverse $E: B^{2,2}_\alpha(F)\to H^1(\R^n)$, $i.e.$ $\operatorname{Tr}\circ E=\operatorname{Id}$ on $B^{2,2}_\alpha(F)$.
 \end{theorem}

\subsection{General framework of closed subsets of $\R^n$}\label{ss-Gen}
In the same time it is possible to consider more general measures than $d$-dimensional measures which can describe by their supports a boundary of a domain~\cite{ARXIV-CREO-2018,JONSSON-1994,JONSSON-2009}.

We follow \cite[Section 1]{JONSSON-1994} and say that a Borel measure $\mu$ on $\mathbb{R}^n$ 
with support $\operatorname{supp} \mu=F$ \emph{satisfies the $D_s$-condition} for an exponent $0<s\le n$ if there is a constant $c_s>0$ such that 
  \begin{equation}\label{EqMuDs}
    \mu(B_{kr}(x))\le c_s k^s \mu(B_r(x)), \quad x\in F, \quad r>0,\quad k\ge 1, \quad 0<k r\le 1.
  \end{equation}
Here as previously $B_r(x)\subset \R^n$  denotes an open ball centered at $x$ and of radius~$r$.
We say that $\mu$ \emph{satisfies the $L_d$-condition} for an exponent $0\leq d\le n$ if for some constant $c>0$ we have 
\begin{equation}\label{EqMuLd}
    \mu(B_{kr}(x))\ge c_d k^d \mu(B_r(x)), \quad x\in F, \quad r>0,\quad k\ge 1, \quad 0<k r\le 1.
  \end{equation}
We also introduce so called  the normalization condition
\begin{equation}\label{Eqnormalized}
c_1\leq \mu(B_1(x))\leq c_2,\quad x\in F,
\end{equation}
where $c_1>0$ and $c_2>0$ are constants independent of $x$.

Combining (\ref{EqMuDs}) and (\ref{Eqnormalized}) one can find a constant $c>0$ such that 
\begin{equation}\label{E:lowreg}
\mu(B_r(x))\geq c\:r^s, \quad x\in F,\quad  0<r\leq 1,
\end{equation}
what implies $\dim_H F \leq s$, where $\dim_H F$ denotes the Hausdorff dimension of $F$. Similarly  (\ref{EqMuLd}) and (\ref{Eqnormalized}) yield a constant $c'>0$ such that 
\begin{equation}\label{E:upreg}
\mu(B_r(x))\leq c'\:r^d, \quad x\in F, \quad 0<r\leq 1,
\end{equation}
hence $\dim_H F \geq d$. Moreover, (\ref{EqMuDs}) implies the doubling condition 
$$\mu(B_{2r}(x))\leq c\:\mu(B_r(x)), \quad x\in F, \quad 0<r\leq 1/2,$$ where $c>0$ is a situable constant, \cite[Section 1]{JONSSON-1994}. 

If a Borel measure $\mu$ with support $F$ satisfies (\ref{E:lowreg}) and (\ref{E:upreg}) with $s=d$ for some $0<d\leq n$, then, according to Definition~\ref{Defdset}, $\mu$ is called a \emph{$d$-measure} and $F$ is called a \emph{$d$-set}.
Obviously, if we have (\ref{EqMuDs}), (\ref{EqMuLd}) and (\ref{Eqnormalized}) and $d=s$ then $\mu$ is a $d$-measure and $F$ a $d$-set. Otherwise, we consider measures, which by (\ref{E:lowreg}) and (\ref{E:upreg}) satisfy for some constants $c>0$ and $c'>0$
\begin{equation}
	c\:r^s\leq \mu(B_r(x))\leq c'\:r^d, \quad x\in F, \quad 0<r\leq 1.
\end{equation}

For this general measure $\mu$ supported on a closed subset $F\subset \mathbb{R}^n$ it is possible thanks to~\cite{JONSSON-1994}  to define the corresponding Lebesgue spaces $L_p(F,\mu)$ and Besov spaces $B_\beta^{p,p}(F,\mu)$ on closed subsets $F\subset \mathbb{R}^n$ in a such way that we have the following theorem 
\begin{theorem}\label{ThGBesov}
Let $0\leq d\leq n$, $d\leq s\leq n$, $s>0$, $1\le p \le +\infty$, 
\begin{equation}\label{E:tracecond}
\frac{n-d}{p}<\beta <1+\frac{n-s}{p},
\end{equation}
and let $F\subset \R^n$ be a closed set which is the support of a Borel measure $\mu$ satisfying (\ref{EqMuDs}), (\ref{EqMuLd}) and (\ref{Eqnormalized}).

 Then, considering the Besov space $B_\beta^{p,p}(F,\mu)$ on $F$, defined as the space of $\mu$-classes of real-valued functions $f$ on $F$ such that the norm 
\begin{multline}
\left\|f\right\|_{B_\beta^{p,p}(F,\mu)}:=\notag\\
\left\|f\right\|_{L_p(F,\mu)}+\left(\sum_{\nu=0}^\infty 2^{\nu(\beta-\frac{n}{p})}\int\int_{|x-y|<2^{-\nu}}\frac{|f(x)-f(y)|^p}{\mu(B(x,2^{-\nu}))\mu(B(y,2^{-\nu}))}\mu(dy)\mu(dx)\right)^{1/p}
\end{multline}
is finite, the following statements hold:  
  \begin{enumerate}
   \item[(i)] $\operatorname{Tr}_F$ is a continuous linear operator from $W^\beta_p(\R^n)$ onto $B^{p,p}_\beta(F)$, and
\begin{equation}\label{E:traceopbound}
\left\|\operatorname{Tr}_F f\right\|_{B_\beta^{p,p}(F)}\leq c_\beta\left\|f\right\|_{W^\beta_p(\mathbb{R}^n)},\quad f\in W^\beta_p(\mathbb{R}^n),
\end{equation}
with a constant $c_\beta>0$ depending only on $\beta$, $s$, $d$, $n$, $c_s$, $c_d$ $c_1$ and $c_2$.
   \item[(ii)] There is a continuous linear extension operator $E_F:B^{p,p}_\beta(F)\to  W^\beta_p (\R^n)$ such that $\operatorname{Tr}_F(E_F f)=f$ for $f\in B^{p,p}_\beta(F)$.
  \end{enumerate}
\end{theorem}
Theorem~\ref{ThGBesov} is a particular case of~\cite[Theorem 1]{JONSSON-1994}. 

The spaces $B_\beta^{p,p}(F,\mu)$ are Banach spaces, while $B_\beta^{2,2}(F,\mu)$ are Hilbert spaces, and their corresponding scalar product is denoted by $\left\langle \cdot,\cdot\right\rangle_{B_\beta^{2,2}(F,\mu)}$. 

A priori the definition of $B_\beta^{p,p}(F,\mu)$ depends on both $F$ and $\mu$. However,  it was shown in \cite[Section 3.5]{JONSSON-1994} that for two different measures $\mu_1$ and $\mu_2$ satisfying hypotheses of Theorem~\ref{ThGBesov} and with common support $F$, if $f\in B^{p,p}_\beta(F,\mu_2)$, then $f$ can be altered on a set with $\mu_2$-measure zero, in such a way that $f$ becomes a function in $B^{p,p}_\beta(F,\mu_1)$. In other
 words, also by Theorem~\ref{ThGBesov}, the spaces  $B_\beta^{2,2}(F,\mu_1)$ and $B_\beta^{p,p}(F,\mu_2)$ are equivalent. Thus, we simplify the notations and instead of $B^{p,p}_\beta(F,\mu)$ simply write $B^{p,p}_\beta(F)$.

Let us notice~\cite{JONSSON-1994} that this time if $F$ is a $d$-set with $0<d\le n$ as defined in Sub-section~\ref{ss-Dset}, then $\mu=m_d$ satisfies (\ref{EqMuDs}), (\ref{EqMuLd}) and (\ref{Eqnormalized}) and hence it is possible to apply Theorem~\ref{ThGBesov}. The restriction on $\beta$ in Theorem~\ref{ThGBesov} becomes $0<\alpha<1$ with $\alpha=\beta-\frac{n-d}{p}$.
Consequently, from one hand, the space $B^{p,p}_\beta(F)$ is equivalent to the Besov space $B^{p,p}_\alpha(F)$ with $0<\alpha<1$ from Sub-section~\ref{ss-Dset} (see~\cite{JONSSON-1984}), which, from the other hand, by our previous remark for $\alpha<1$, explains why we don't need to impose  that $F$ preserves the local Markov inequality.   
 Thus in the framework of $d$-sets this theorem coincides with Theorem~\ref{ThJWF} for $\alpha<1$. 
 
 \begin{remark}\label{remH1-B1}
 	If we apply Theorem~\ref{ThGBesov} for $H^1(\R^n)$ we obtain the image of the trace equal to the Hilbert space $B^{2,2}_1(F)$ with the restrictions $$ n\ge s\ge d>n-2\ge 0.$$
 \end{remark}


\subsection{Integration by parts and the Green formula}\label{ss-Green}
Let us generalize the Green formula formulated for $d$-sets in~\cite{ARFI-2017} (initially proposed by Lancia~\cite[Thm.~4.15]{LANCIA-2002} for a Von Koch curve) and the integration by parts from Appendix~A Theorem~A.3~\cite{MAGOULES-2017} (see also the proof of formula~(4.11) of Theorem~4.5 in~\cite{ARXIV-CREO-2018}).

\begin{proposition}\label{PropGreen}\textbf{(Green formula)}
 Let $\Omega$ be a  domain in $\R^n$ ($n\ge 2$) with a closed boundary $\del \Omega$ which is the support of a Borel measure $\mu$ satisfying the conditions of Theorem~\ref{ThGBesov}  with $ n\ge s\ge d>n-2\ge 0.$
 Then
 \begin{enumerate}
 	\item the Green formula holds for all $u$ and $v$ from $H^1(\Omega)$ with $\Delta u\in L_2(\Omega)$,
\begin{equation}\label{FracGreen}
 \int_\Omega v\Delta u\dx + \int_\Omega \nabla v\cdot \nabla u \dx=\langle \frac{\del u}{\del n}, 
\mathrm{Tr}v\rangle _{((B^{2,2}_{1}(\del \Omega))', B^{2,2}_{1}(\del \Omega))},
\end{equation}
where $(B^{2,2}_{1}(\del \Omega))'$  is the dual space of $B^{2,2}_{1}(\del \Omega)$.
\item In addition the usual integration by parts holds for all $u$ and $v$ from $H^1(\Omega)$ in the following weak sense
      \begin{equation}\label{IPP}
 \langle u \nu_i,v\rangle_{(B^{2,2}_{1}(\del \Omega))', B^{2,2}_{1}(\del \Omega))}:= \int_\Omega \frac{\del u}{\del x_i} v\dx+\int_{\Omega} u\frac{\del v}{\del x_i}\dx \quad i=1,\ldots,n,
\end{equation}
where by $u \nu_i$ is denoted the linear continuous functional on $B^{2,2}_{1}(\del \Omega)$.
 \end{enumerate}
\end{proposition}
\begin{proof}
The statement follows, thanks to Theorem~\ref{ThGBesov}, from the surjective property of the linear continuous trace operator $\mathrm{Tr}_{\del \Omega}:H^1(\Omega)\to B^{2,2}_{1}(\del \Omega)$.
Let us prove~(\ref{FracGreen}). We define  for a fixed $u\in H^1(\Omega)$ with $\Delta u\in L_2(\Omega)$ the linear functional
$$L: \quad w \in B^{2,2}_{1}(\del \Omega) \mapsto L(w)= \int_\Omega E_{\del \Omega}w\Delta u\dx + \int_\Omega \nabla ( E_{\del \Omega}w)\cdot \nabla u \dx \in \C,$$
where by Theorem~\ref{ThGBesov} the operator $E_{\del \Omega}: B^{2,2}_{1}(\del \Omega) \to H^1(\Omega)$ is the linear bounded extension operator such that $\operatorname{Tr}_{\del \Omega}(E_{\del \Omega} w)=w$ $\mu$ a.e. Hence, we denote $E_{\del \Omega}w\in H^1(\Omega)$ by $v$.
Then, by the continuity of the extension operator, we state that $L$ is continuous:
\begin{multline*}
	|L(w)|\le \left|\int_\Omega v\Delta u\dx \right|+\left|\int_\Omega \nabla v\cdot \nabla u \dx\right|\\
	\le\left (\|\Delta u\|_{L_2(\Omega)}+\|\nabla u\|_{L^2(\Omega)}\right)\|v\|_{H^1(\Omega)}\le C\left(\|\Delta u\|_{L_2(\Omega)}+\|\nabla u\|_{L^2(\Omega)}\right)\|w\|_{B^{2,2}_{1}(\del \Omega)}.
\end{multline*}
Hence, as $\operatorname{Tr}_{\del \Omega}(v)=w$, with notations $L(w)=\langle \frac{\del u}{\del n}, 
\mathrm{Tr}v\rangle _{((B^{2,2}_{1}(\del \Omega))', B^{2,2}_{1}(\del \Omega))}$ we exactly obtain the generalized Green formula~(\ref{FracGreen}).
By the same argument we also have~(\ref{IPP}).
%
%
%
%

\end{proof}

\section{Sobolev admissible domains and the generalization of the Rellich-Kondrachov theorem}\label{SecTh-RK}
Thanks to Theorems~\ref{ThHajlasz} and~\ref{ThGBesov} we can generalize now the notion of admissible domains introduced in~\cite{ARFI-2017} in the framework of $d$-sets:
\begin{definition}\textbf{(Sobolev admissible domain)}\label{DefAdmis} Let $1<p<\infty$ and $k\in \N^*$ be fixed.
 A domain $\Omega\subset \R^n$  is called a ($W_p^k$-) Sobolev  admissible domain if it is an $n$-set, such that  $W_p^k(\Omega)=C_p^k(\Omega)$  as sets with equivalent norms (hence, $\Omega$ is a $W_p^k$-extension domain), with a closed  boundary $\del \Omega$ which is the support of  a Borel measure $\mu$ satisfying the conditions of Theorem~\ref{ThGBesov}.
\end{definition}

 Therefore, we summarize useful in  what follows results (see~\cite{SHVARTSMAN-2010} for more general results for the case $p>n$) 
 for the trace and the extension operators :
\begin{theorem}\label{ThContTrace}
 Let $1<p<\infty$, $k\in \N^*$ be fixed. Let $\Omega$ be a Sobolev admissible domain in $\R^n$.
Then the following trace operators (see Definition~\ref{DefGTrace})
\begin{enumerate}
 \item $\mathrm{Tr}: W_p^k(\R^n)\to B^{p,p}_{k}(\del \Omega)\subset L_p(\del \Omega)$,
 \item $\mathrm{Tr}_\Omega:W_p^k(\R^n)\to W_p^k(\Omega)$,
 \item $\mathrm{Tr}_{\del \Omega}:W_p^k(\Omega)\to B^{p,p}_{k}(\del \Omega)$
\end{enumerate}
are linear continuous and surjective with linear bounded  right inverse, $i.e.$ extension, operators $E: B^{p,p}_{k}(\del \Omega)\to W_p^k(\R^n)$, $E_\Omega: W_p^k(\Omega)\to W_p^k(\R^n)$ and $E_{\del \Omega}: B^{p,p}_{k}(\del \Omega)\to W_p^k(\Omega)$.
\end{theorem}
\begin{proof}
 It is a corollary of results given in Sections~\ref{SecExtDom} and~\ref{secTraceGreen}. 
 Indeed, if $\Omega$ is Sobolev admissible, then by Theorem~\ref{ThGBesov}, the trace operator $\mathrm{Tr}: W_p^k(\R^n)\to B^{p,p}_{k}(\del \Omega)\subset L_p(\del \Omega)$ is linear continuous and surjective with linear bounded  right inverse $E: B^{p,p}_{k}(\del \Omega)\to W_p^k(\R^n)$ (point 1). On the other hand, by~\cite{HAJLASZ-2008}, $\Omega$ is a $W_p^k$-extension domain and $\mathrm{Tr}_\Omega:W_p^k(\R^n)\to W_p^k(\Omega)$ and $E_\Omega: W_p^k(\Omega)\to W_p^k(\R^n)$ are linear continuous (point 2). Hence, the embeddings
 $$B^{p,p}_{k}(\del \Omega) \to W_p^k(\R^n)\to W_p^k(\Omega) \quad \hbox{and} \quad W_p^k(\Omega)\to W_p^k(\R^n)\to B^{p,p}_{k}(\del \Omega)$$
 are linear continuous (point 3).
\end{proof}
By updating the class of admissible domains, all results of~\cite{ARFI-2017} still hold in this new class. For instance it is also possible to consider Sobolev admissible truncated domains for which two disjoint boundaries satisfy~Theorem~\ref{ThGBesov}.
Without any particular motivation here for the truncated domain, let us just formulate  the compactness of the embedding $H^1(\Omega)$ to $L^2(\Omega)$ for the Sobolev admissible domains:

 \begin{proposition}\label{PropCompEmb}
  Let
  $\Omega$ be a bounded Sobolev admissible truncated domain for $p=2$ and $k=1$.
  Then  the Sobolev space $H^1(\Omega)$ is compactly embedded in $L_2(\Omega)$: $$H^1(\Omega) \subset \subset L_2(\Omega).$$
 \end{proposition}
 
 The proof follows with small modifications the proof of Proposition~2 in~\cite{ARFI-2017} and thus is omitted.
 
 However, we would like to recall main compactness results of~\cite{ARFI-2017} putting them in the new framework of Sobolev admissible domains with not necessarily a $d$-set boundary. 
 
 \begin{remark}
To have a compact embedding  it is important  that
 the domain $\Omega$ be a $W^k_p$-extension domain.
The boundness or unboudness of $\Omega$ is not important to have $W_p^{k}(\Omega)\subset \subset W_p^{\ell}(\Omega)$ with $k>\ell\ge1$ ($1<p<\infty$).
But the boundness of $\Omega$ is important for the  compact embedding in $L_q(\Omega)$.
\end{remark}
As a direct corollary we have the following generalization of the classical Rellich-Kondrachov theorem (see for instance Adams~\cite{ADAMS-2003} p.144 Theorem 6.2):
\begin{theorem}\label{ThCSEnSET} \textbf{(Compact Sobolev embeddings for $n$-sets, \cite{ARFI-2017})}
 Let $\Omega\subset \R^n$ be
 an $n$-set with $W_p^k(\Omega)=C_p^k(\Omega)$, $1<p<\infty$, $k,\ell\in \N^*$. Then there hold  the following compact  embeddings:
    \begin{enumerate}
     \item $W_p^{k+\ell}(\Omega)\subset \subset W_q^\ell(\Omega)$,
     \item $W_p^{k}(\Omega)\subset \subset L_q^{loc}(\Omega)$,  or $W_p^{k}(\Omega)\subset \subset L_q(\Omega)$ if $\Omega$ is bounded,
    \end{enumerate}
 with $q\in[1,+\infty[$ if $kp=n$, $q\in [1,+\infty]$ if $kp>n$, and with $q\in[1, \frac{pn}{n-kp}[$ if $kp<n$. 
\end{theorem}
\section{Compactness of the trace}\label{secCompTr}

In the same way as in~\cite{ARFI-2017}, we generalize  the classical Rellich-Kondrachov theorem for fractals:
\begin{theorem}\label{ThAnnaEmbF} \textbf{(Compact Besov embeddings)}
 Let $F\subset \R^n$ be a closed set satisfying conditions of Theorem~\ref{ThGBesov}, 
 $\beta=k+\ell>0$ for $k,\ell\in \N^*$.

 Then, for the same $q$ as in Theorem~\ref{ThCSEnSET},
 the following continuous embeddings  are compact
 \begin{enumerate}
  \item $B^{p,p}_{\beta}(F)\subset \subset B^{q,q}_\ell(F)$ for  $\ell\ge 1$; 
  \item if $F$ is bounded in $\R^n$, $B^{p,p}_{\beta}(F)\subset \subset L_q(F)$, otherwise $B^{p,p}_{\beta}(F)\subset \subset L_q^{loc}(F)$ for $\ell\ge 0$.
 \end{enumerate}
 \end{theorem}
\begin{proof}
Indeed, for $\beta=k+\ell$, thanks to Theorem~\ref{ThGBesov}, the extension
$E_F:B^{p,p}_{\beta}(F) \to W_p^{\beta}(\R^n)$ is continuous. Hence, by Calderon~\cite{CALDERON-1961}, a non trivial ball is $W_p^{\beta}$-extension domain:
$\mathrm{Tr}_{B_r}$ (see the proof of Theorem~\ref{ThCSEnSET} in~\cite{ARFI-2017}) is continuous. Thus, the classical Rellich-Kondrachov theorem on the ball $B_r(x)$ gives the compactness of $K: W_p^{\beta}(B_R)\to W_q^\ell(B_r)$. Since, for $\ell\ge 1$, $E_2:W_q^\ell(B_r)\to W_q^\ell(\R^n)$ is continuous and, by Theorem~\ref{ThJWF},
 $\mathrm{Tr}_F:W_q^\ell(\R^n)\to B^{q,q}_\ell(F)$ is continuous too, we conclude that the operator $$\mathrm{Tr}_F\circ K \circ \mathrm{Tr}_{B_r} \circ E_F:B^{p,p}_{\beta}(F) \to B^{q,q}_\ell(F)$$
 is compact. For $j=0$, we have $W_q^0=L_q$, and hence, if $F\subset B_r(x)$, the operator
 $ L_q(B_r(x)) \to L_q(F)$ is a linear continuous measure-restriction operator in the sense of measure $\mu$. 
 If $F$ is not bounded in $\R^n$,  for all bounded $\mu$-measurable subsets $K$ of $F$, the embedding $L_q(B_r(x))\to L_q(K)$ is continuous.
 \end{proof}

In particular, the compactness of the trace operator implies the following equivalence of the norms on $W_p^k(\Omega)$:
 \begin{proposition}\label{PropCompET}
 Let $\Omega$ be a Sobolev admissible domain in $\R^n$ with a compact boundary $\del \Omega$ and $1<p<\infty$, $k\in \N^*$. Then
 \begin{enumerate}
  \item $W_p^k(\Omega)\subset \subset L_p^{loc}(\Omega)$;
  \item $\mathrm{Tr}: W_p^k(\Omega)\to L_p(\del \Omega)$ is compact;
  \item If in addition the mesure $\mu$ is Borel regular then the image
       $\operatorname{Im}(\mathrm{Tr})= B^{p,p}_{k}(\del \Omega)$ is dense in $L_p(\partial\Omega)$. 
  \item $\|u\|_{W_p^k(\Omega)}$ is equivalent to $\|u\|_{\mathrm{Tr}}=\left(\sum_{|l|=1}^k \int_\Omega |D^l u|^p\dx +\int_{\del \Omega} |\mathrm{Tr}u|^pd\mu \right)^\frac{1}{p}. $
 \end{enumerate}
\end{proposition}
\begin{proof}
Let us prove point~3. To prove all other points it is sufficient to follow the proof of Proposition~2 in~\cite{ARFI-2017}.

If $\del \Omega$ is endowed with a Borel regular measure $\mu$, then
 the space $\{v|_{\del \Omega} : v \in \mathcal{D}(\R ^n)\}$, which is dense in $C(\del \Omega)$ by the Stone-Weierstrass theorem for the uniform norm,   is also dense in $L_p(\del \Omega)$ (see Theorem~2.11 in~\cite{EVANS-2010}). Hence, $B^{p,p}_k(\del \Omega)$ is dense in $L_p(\del \Omega)$. 
 \end{proof}


The Poincar\'e's inequality stays also true on a bounded Sobolev admissible domain~\cite{DEKKERS-2019}:
\begin{theorem}\label{inegPoinc}(Poincar\'e's inequality)
Let $\Omega\subset\mathbb{R}^n$ with $n\geq 2$ be a bounded connected Sobolev admissible domain. For all $u\in W^{1,p}_0(\Omega)$ with $1\leq p<+\infty$, there exists $C>0$ depending only on $\Omega$, $p$ and $n$ such that
$$\Vert u\Vert_{L^p(\Omega)}\leq C \Vert \nabla u\Vert_{L^p(\Omega)} .$$
Therefore the semi-norm $\Vert .\Vert_{W^{1,p}_0(\Omega)}$, defined by $\Vert u\Vert_{W^{1,p}_0(\Omega)}:=\Vert \nabla u\Vert_{L^p(\Omega)}$, is a norm which is equivalent to $\Vert .\Vert_{W^{1,p}(\Omega)}$ on $W^{1,p}_0(\Omega)$.

Moreover for all $u\in W^{1,p}(\Omega)$ there exists $C>0$ depending only on $\Omega$, $p$ and $n$ such that
$$\left\Vert u-\frac{1}{\lambda(\Omega)}\int_{\Omega} u\;d\lambda\right\Vert_{L^p(\Omega)}\leq C \Vert \nabla u\Vert_{L^p(\Omega)} .$$
\end{theorem}
\begin{proof}
The result for $u\in W^{1,p}_0(\Omega)$ comes from the boundness of $\Omega$. The result for $u\in W^{1,p}(\Omega)$ comes from the compactness of the embedding $W^{1,p}(\Omega)\subset\subset L^p(\Omega)$ from Theorem \ref{ThCSEnSET} and following for instance the proof  from~\cite{EVANS-2010} (see section 5.8.1 Theorem 1).
\end{proof}
Thus the results of~\cite{ARFI-2017}  on the Dirichlet-to-Neumann operator can be also updated in the framework of the Sobolev admissible domains of Definition~\ref{DefAdmis}.
For instance we have
\begin{theorem}\label{PoincStek}
Let $\Omega$ be a bounded Sobolev admissible domain in $\mathbb{R}^n$ ($n\geq2$) for $p=2$ and $k=1$. Then  the Poincaré-Steklov operator
$$A:B^{2,2}_{1}(\partial\Omega) \rightarrow (B^{2,2}_{1}(\partial\Omega))'$$
mapping $u\vert_{\partial\Omega}$ to $\partial_{\nu}u\vert_{\partial\Omega}$ is a linear bounded self adjoint operator with $\ker A \neq 0$.
\end{theorem}
\section{Application to the Poisson boundary valued and spectral problems}\label{secApplic}
In this section we show the application of the theory of functional spaces developed in the previous sections on the example of the Poisson equation with Robin boundary conditions which we can weakly solve on the Sobolev admissible domains. 
Let $\Omega$ be a $H^1$-Sobolev admissible domain with a compact boundary $\del \Omega$ and  $f\in L^2(\Omega)$. 
For $a>0$ we define $H^1(\Omega)$ 
endowed with the equivalent by Proposition~\ref{PropCompET} norm
\begin{equation}\label{defH1til}
\Vert u\Vert_{\mathrm{Tr}}^2=\int_{\Omega} |\nabla u|^2 dx+a\int_{\partial\Omega} |Tr_{\partial\Omega}u|^2 d\mu.
\end{equation}
Then $u\in H^1(\Omega)$ 
is called a weak solution of the Poisson problem~(\ref{LaplaceeqRobin1}) if for for all $v\in H^1(\Omega)$
$$ (u,v)_{\mathrm{Tr}}=\int_{\Omega} \nabla u\nabla v \;dx+a\int_{\partial\Omega} Tr_{\partial\Omega}uTr_{\partial\Omega}v \;d\mu=\int_{\Omega} fv\;dx.$$
Thus, the Riesz representation theorem gives us the well-posedness result:
\begin{theorem}\label{thmwplaplrob}
Let $\Omega$ be a $H^1$-Sobolev admissible domain with a compact boundary $\del \Omega$. Then for all  $f\in L^2(\Omega)$ and $a>0$ there exists a unique weak solution $u\in H^1(\Omega)$ of the Poisson problem~(\ref{LaplaceeqRobin1}) and it holds the stability estimate
$$\Vert u\Vert_{\mathrm{Tr}} \leq C \Vert f\Vert_{L^2(\Omega)}.$$
\end{theorem}
In the same time with the additional assumption that $\Omega$ is bounded,  ensuring the compactness of the embedding $i_{L_2(\Omega)}:H^1(\Omega)\to L_2(\Omega)$ by Proposition~\ref{PropCompET}, we also have the compactness of the operator $B: f\in L_2(\Omega) \mapsto B(f)=u\in H^1(\Omega)$ mapping a source term $f$ to the weak solution of the Poisson problem~(\ref{LaplaceeqRobin1}) (see for instance Theorem~3.6~\cite{ARFI-2017}). The compactness of the embedding $i_{L_2(\Omega)}$ allows also to apply the spectral Hilbert-Schmidt theorem for a auto-adjoint compact operator on a Hilbert space to obtain the usual properties of the spectral problem for the $-\Delta$ on the Sobolev admissible domains:
\begin{theorem}\label{thmeigenfuncLaprob}
Let $\Omega$ be a bounded $H^1$-Sobolev admissible domain. The weak eigenvalue problem
$$\forall v\in H^1(\Omega) \quad (u,v)_{\mathrm{Tr}}=\lambda (u,v)_{L_2(\Omega)}$$
has a countable number of strictly positive eigenvalues of finite multiplicity, which is possible to numerate in the non-decreasing way:
$$0<\lambda_1\leq\lambda_2\leq\lambda_3\leq\cdots, \quad \lambda_j\to +\infty \quad j\to +\infty.$$
In addition the corresponding eigenfunctions forms  an orthonormal basis of $L^2(\Omega)$ and an orthogonal basis of $H^1(\Omega)$.
\end{theorem}
\begin{proof}
It is sufficient to notice that the eigenvalue problem is equivalent to the spectrum problem  $Tu=\frac{1}{\lambda} u$ for the operator
$T=A\circ i_{L_2(\Omega)}: H^1(\Omega)\to H^1(\Omega)$ which is linear compact and auto-adjoint on the Hilbert space $H^1(\Omega)$.
Here $A$ is the linear bounded operator (existing by the Riesz representation theorem) which maps $v\in L_2(\Omega)$ to $Av\in H^1(\Omega)$ such that
$$\forall \phi \in H^1(\Omega) \quad (v,\phi)_{L_2(\Omega)}=(Av,\phi)_{\mathrm{Tr}}.$$
\end{proof}
 \vspace*{4pt}

 \def\refname{References}
\bibliographystyle{siam}
\label{bib:sec}
\bibliography{/home/anna/Documents/Statii/bibtex/biblio.bib}
\end{document}